\documentclass[reqno,10pt,a4paper]{amsart}
\usepackage{amsmath,amssymb,exscale, amscd}
\usepackage[utf8]{inputenc}
\usepackage{amsmath}
\usepackage[all]{xy}
\usepackage[english]{babel}
\usepackage{mathrsfs}
\usepackage{mathalfa}
\usepackage{graphicx}
\usepackage{tikz}
\usepackage{epigraph}
\usepackage{setspace}
\usepackage{enumitem}
\usepackage[utf8]{inputenc}
\usepackage{amsthm}
\usepackage{amssymb}
\usepackage{a4wide}

\newtheorem{theorem}{Theorem}[section]
\newtheorem{lemma}[theorem]{Lemma}
\newtheorem{corollary}[theorem]{Corollary}
\newtheorem{proposition}[theorem]{Proposition}

\theoremstyle{definition}
\newtheorem{definition}[theorem]{Definition}
\newtheorem{remark}[theorem]{Remark}

\newtheorem{subsct}[theorem]{}
\theoremstyle{plain}

\makeatletter
\newcommand{\extp}{\@ifnextchar^\@extp{\@extp^{\,}}}
\def\@extp^#1{\mathop{\bigwedge\nolimits^{\!#1}}}
\makeatother
\font\smallrm=cmr7

\author[
\smallrm{E.~Esteves, N.~Medeiros, W.~Sousa}
]
{Eduardo Esteves, Nivaldo Medeiros \and W\'{a}llace Sousa}
\title[Limits of dual curves]{Limits of dual curves via foliations}

\thanks{First author supported by CNPq, Proc.~304623/2015-6 and
    FAPERJ, Proc.~E-26/202.992/2017. Third author partially supported
    by CAPES, Finance Code 001.}


\begin{document}

\begin{abstract}
We develop a method to compute limits of
dual plane curves in Zeuthen families of any kind.
More precisely, we compute the limit
$0$-cycle of the ramification scheme of a general linear system on the
generic fiber, only assumed geometrically reduced, of a Zeuthen family of any kind.
\end{abstract}

\maketitle

\section{Introduction}\setcounter{equation}{0}

\begin{subsct}
{\bf The problem.} Let $C(t)$ be a family of
projective plane curves degenerating to $C:=C(0)$. More precisely, 
consider the one-parameter family $C(t)$ of plane curves given by a
convergent \emph{homogeneous
power series}
$$
F(t):= F_0 + F_1t + F_2t^2 + \cdots + F_it^i + \cdots,
$$
with $F_i \in \mathbb{C}[X_0,X_1,X_2]$ homogeneous of the same
degree and $F_0 \neq 0$. Suppose that for $t \in
\mathbb{C}$ near 0 the plane curve $C(t)$ is nonsingular. We may ask
which plane curve the dual curve to $C(t)$ degenerates to as $t$
approaches $0$. In other words, what is the limit of the dual plane curves of the family?

The present article addresses this question, which is not new. Indeed,
the history of the problem goes back to at least the 19th Century,
surfacing in works by Maillard \cite{M} and Zeuthen \cite{Z1},
\cite{Z2}. They worked on computing  limits of dual curves for
certain one-parameter families of plane curves $C(t)$ as a step in the
determination of characteristic numbers of plane curves.

Characteristic numbers are basic
enumerative invariants. They answer the question: How many smooth
plane curves of a given degree $d$ pass through $a$ general points and
are tangent to $b$ general lines, for $a+b=d(d+3)/2$? For $d=2$ the
numbers 
are classical, obtained 
through the moduli of complete conics, a blowup of projective 5-space
along the Veronese surface; see \cite{Kl2} for a history.

Zeuthen predicted those numbers for $d=3, 4$. For this Zeuthen used
certain families $C(t)$ which he called of ``first kind'', ``second
kind'' and ``third kind''  (see Remark~\ref{5.1.4}), observing that
for them the limit of the dual curves depends only on the first few terms of
the power series expansion of $F(t)$. 

More recently, van Gastel \cite{vG} computed limits of conormals of plane
curves, following the theory on the conormal scheme developed by
Kleiman \cite{Kl3}, with the same purpose as Maillard and Zeuthen, to compute
characteristic numbers. Also, Katz \cite{Ka} computed limits of dual
curves by using Newton polygons, for families satisfying certain
\textit{regularity} conditions. 

For $d=3$ the characteristic numbers were rigorously computed
by Aluffi \cite{A1} and Kleiman and Speiser \cite{KlS}. And for $d=4$
most of them were computed by Aluffi \cite{A2} and van Gastel \cite{vG},
and the remaining by Vakil \cite{V} using Kontsevich's moduli space of
stable maps. For $d=5$ and above most characteristic numbers are not computed.

The present paper defines Zeuthen families of type $n$ for every
natural number $n$ (following van Gastel's definition, which is
different from Zeuthen's 
for $n = 4$) and introduces a new approach to computing limits of dual curves, and more generally limits of ramification points. With our method we are able to compute for instance limits of dual curves for Zeuthen families of the ``third kind'', which are not \textit{regular} in
general, in Katz's terminology. More generally, we consider families
of curves $C(t)$ given by homogeneous power series of the form
$$
F(t)=E^2A+F_1t+F_2t^2+\cdots,
$$
where $E$ and $A$ are square-free and coprime. If $C(t)$ is
generically reduced, we give a formula for the limit of the dual
curves of the family $C(t)$, our Corollary~\ref{5.1.6}.

Our argument is algebraic. We may replace
$\mathbb C$ by any algebraically closed field of characteristic 0. 

We do not compute characteristic numbers. It would be a natural
endeavor to apply the techniques developed here to compute new
characteristic numbers, but we suspect that the work ahead is still substantial.
\end{subsct}

\begin{subsct}\setcounter{equation}{0}
{\bf The method.} Let $C$ be a projective plane curve defined over an
algebraically closed field $k$ of characteristic $p\geq 0$. If $C$ is
smooth, to describe its dual curve we may consider the \textit{ramification
  schemes} $R_C(V)$ associated to linear systems $V$ on $C$. 
More precisely, for each $P\in C$ and each nonnegative integer
$\ell$, let 
$V(-\ell P) \subseteq V$ be the linear subsystem of sections of $V$
vanishing at $P$ with multiplicity at least $\ell$. We call $P$ a 
\textit{ramification point} of $V$ if $V(-(\dim V)P) \not = 0$. 
We can see the set of ramification points of $V$ as a subscheme of
$C$. In fact, this scheme can be computed locally as the locus cut out
on $C$ by a ``Wronskian'' curve, even if $C$ is singular; see 
Subsection~\ref{ramifsch}. 
It is this scheme that we denote by $R_C(V)$.

Given a general point $R \in \mathbb{P}^2_k$, consider the
ramification subscheme of $C$ associated to the linear
system $V_R$ cut out by the lines passing through $R$, so that a
simple point $P$ of $C$ is in the
support of the subscheme if and only if the line $\overline{RP}$ is
tangent to $C$ at $P$. If $C$ is smooth then the subscheme is a
Cartier divisor and the dual curve $C^\vee$ satisfies:
\begin{equation}\label{polar}
C^\vee \cap R^\vee = \sum_P n_P \overline{RP}^\vee,
\end{equation}
where $n_P$ is the multiplicity of $P$ in the ramification divisor.
This implies that the dual curve of a smooth plane curve is completely 
determined by ramification schemes. So we may, in principle, compute limits of dual curves by
computing limits of ramification schemes.

If $C$ is singular, but without multiple components, its dual curve is usually defined
using only its smooth locus, as the closure in the dual plane of the locus of tangent
lines to simple points of $C$. It is however better, for our purposes,
to adopt a different definition, that in \cite{EH}, p.~240. 
With that definition, the dual curve $C^\vee$ is made
up of the usual dual curves of the components of $C$, each with
multiplicity 1, and the lines dual to the singular
points $P\in C$, each with multiplicity $n_P$ equal to the intersection
multiplicity of a general polar with $C$ at $P$. Thus, if $C=C(0)$ for
a family $C(t)$, then $C^\vee$ is the limit of the dual curves
of this family. Also, Equation~\eqref{polar} holds!

If $C$ has multiple components, and $C=C(0)$ for a family $C(t)$, the
limit of the dual curves depends on the family $C(t)$. To compute the
limit, we compute the limit of the ramification divisors of the linear
systems cut out on the family by the lines passing through a general
point $R \in \mathbb{P}^2_k$. (In other words, we describe
the limit of the dual curves by describing its intersection with
a general line in the dual plane.) And to compute limits of
ramification divisors we resort to foliations.

A (singular) foliation of the projective plane is a rank-1 subsheaf of
the tangent bundle $T_{\mathbb{P}^2_k}$. In more concrete
terms, a foliation is associated to a homogeneous derivation of
$k[X_0,X_1,X_2]$, that is, a derivation
$$
\partial=G_0\partial_{X_0} + G_1\partial_{X_1} + G_2\partial_{X_2},
$$
where $G_0,G_1,G_2 \in k[X_0,X_1,X_2]$ are homogeneous of the same
degree. If $C$ is given by $F=0$, we say the foliation leaves $C$ invariant, 
or that $\partial$ is a $F$-derivation, if
$F|\partial(F)$. That is the case for instance
if $\partial=\partial_{F,H}$ for any homogeneous polynomial $H$; see
Subsection~\ref{wronskians}. 

If $V$ is a linear system on $\mathbb{P}^2_k$ 
given by homogeneous polynomials in
$k[X_0,X_1,X_2]$ of the same degree, we can
use a $F$-derivation $\partial$ to compute ramification. More
precisely, we can consider the so-called \emph{extatic} curve $W_\partial(V)=0$ of
the foliation, defined in \cite{P}. The polynomial $W_\partial(V)$ is the Wronskian
determinant of a basis of $V$ with respect to $\partial$; see
Subsection~\ref{wronskians}.  If the ramification scheme, $R_C(V)$, of the
linear system cut out on $C$ by $V$ is a Cartier divisor and
$\partial=\partial_{F,H}$ for $H$ prime to $F$, then our 
Lemma~\ref{2.3.2} implies that $W_\partial(V)=0$ cuts out on $C$ the divisor
$R_C(V)$ plus $\binom{r+1}{2}$ times the divisor cut out by $H=0$.

Our lemma is the main ingredient of our first application of our
method, Proposition~\ref{3.3.2}, as we explain now.

If $C=C(0)$ for a family of curves $C(t)$ given by a homogeneous power
series $F(t)=\sum F_it^i$, we want to consider a family of foliations given by a
family of derivations 
$$
\partial(t)=G_0(t)\partial_{X_0} + G_1(t)\partial_{X_1} + G_2(t)\partial_{X_2},
$$
where the $G_i(t)$ are homogeneous power series of the same degree. We
want to choose a $\partial(t)$ that is a \emph{$F(t)$-derivation},
that is, $F(t)|\partial(t)(F(t))$. For instance,
$\partial(t):=\partial_{F(t),H}$ for $H$ prime to $F$ and
$\partial'(t):=(1/t)\partial_{F_0,F(t)-F_0}$ are $F(t)$-derivations. 
If $C$ has multiple components,
$\partial'(t)$ is a multiple of those components. Factoring them out,
we get a derivation $\partial''(t)$. If $F_1$ is prime to $F_0$, then
$\partial''(0)$ does not vanish on any component of $C$. 
We say $\partial''(t)$ is a \emph{reduction} of $\partial'(t)$; see
Subsection~\ref{3.3}. 

If $V$ is a general linear system on
$\mathbb{P}^2_k$, 
in the sense that all the poynomials defining
it are prime to $F_0$, then the extatic curve $W_{\partial''(0)}(V)=0$
intersects $C(0)$, cutting out a Cartier divisor. We can thus use the family of extatic
curves given by $W_{\partial''(t)}(V)$
to compute the limit of the
ramification divisors associated to $V$ in the family; a formula is
given in Proposition~\ref{3.3.2}. 

Unfortunately though, the condition on $F(t)$ above is too strict. To
be able to compute limits of dual curves in Zeuthen families, we need
more flexibility. For a component of $C=C(0)$ that is not
multiple, $\partial(0)$ does not vanish on that component. It does
vanish on the multiple components of $C$, but one might not be able to
reduce $\partial(t)$ as we were able to reduce $\partial'(t)$. On the
other hand, $\partial''(0)$ vanishes on the common components of
$C$ and $F_1=0$, but we might not be able to reduce
$\partial''(t)$. The flexibility we want is that of choosing for each
component of $C$ a
family of derivations \emph{adapted} to it, work
independently with each family, and compute the limit on each
component of $C$ of the
ramification divisors of $V$ on $C(t)$, by computing the
limit on that component of the intersection of the family $C(t)$ with 
the associated family of extatic curves. 

We develop these ideas in  Section~\ref{section:adaptions}, whose main
result, Theorem~\ref{4.1.4}, relies heavily on a general formula for limits of
Cartier divisors appearing in \cite{E2}. It is this formula that
allows us to put together the limits computed on each component of
$C$ to obtain a global limit, if certain conditions are satisfied.

In Section~\ref{example} we apply Theorem~\ref{4.1.4} 
to compute limits of ramification divisors
for families $C(t)$ that do not satisfy the conditions for
Proposition~\ref{3.3.2}, 
but that include all Zeuthen families of the first
kind. These families are given by homogeneous power series
$F(t)=F_0+F_1t+\cdots$ such that the common factors of $F_0$ and $F_1$
are simple factors of $F_0$. For these families, $\partial(t)$ and $\partial''(t)$ are the
families of derivations needed. Applying Theorem~\ref{4.1.4} we get a
formula for the limit of the ramification divisors of families of general linear
systems along $C(t)$, our Theorem~\ref{5}, generalizing Proposition~\ref{3.3.2}. As a corollary,
we give a formula for the limit of the dual curves of these families; 
see Corollary~\ref{5cor}. 

Finally, we show that we can also apply Theorem~\ref{4.1.4} to compute
limits of ramification divisors of general linear systems 
for Zeuthen families of any kind, our Theorem~\ref{5.1.5}, and as a
corollary we get formulas for the limits of dual curves. 
Here we potentially need more than
two families of derivations.

Many interesting questions arise from our study. First, when can one
apply Theorem \ref{4.1.4} to compute limits of ramification divisors?
When are there families of $F(t)$-derivations adapted to each
component of $C$ and satisfying the conditions stipulated in the
theorem? Second, how to handle nongeneral linear systems, for instance
the system of all lines when $C$ contains one? In this case, the
question is: what are the limits of inflection points along $C(t)$?
\end{subsct}

\begin{subsct}\setcounter{equation}{0}
{\bf Outline.} 
The paper is organized as follows. In Section~\ref{wrs}
we show how to compute ramification schemes of linear 
systems on curves $C$ using Wronskians induced by foliations.
In Section~\ref{families} we extend the construction to families $C(t)$
and give in Section~\ref{example0} 
a formula for the limit of the ramification schemes of families
of general linear systems along $C(t)$, provided
$C(t)$ degenerates to $C(0)$ along a general direction.
In Section~\ref{section:adaptions} we show how limits of ramification
schemes can be computed when one can find adapted families of
derivations, even for nongeneral degenerations; we specify conditions
and give a formula in Theorem~\ref{4.1.4}. In Section~\ref{example} we
use the method of adaptation to generalize the formula we obtained in
Section~\ref{example0}. Finally, in
Section~\ref{section:zeuthen} 
we compute
limits of dual curves for Zeuthen families of any kind, our Corollary~\ref{5.1.6}.
\end{subsct}

\section{Wronskians and ramification schemes}\label{wrs}

\begin{subsct}\label{wronskians}\setcounter{equation}{0}
{\bf Wronskians.} Let $k$ be a ring and $S$ a $k$-algebra. Let $\partial$ be a $k$-derivation of $S$, and $v:=[a_0 \,\, \cdots \,\, a_r]$ a row matrix of elements $a_i \in S$. We say the determinant$$W_{\partial}(v):=\det\begin{bmatrix}
a_0 & a_1 & \cdots & a_r \\
\partial(a_0) & \partial(a_1) & \cdots & \partial(a_r)\\
\vdots & \vdots & \ddots & \vdots\\
\partial^r(a_0) & \partial^r(a_1) & \cdots & \partial^r(a_r)
\end{bmatrix},$$
where $\partial^i$ denotes the $i$-th iteration of $\partial$, is the \textit{Wronskian} of $v$ with respect to $\partial$.

The multilinearity of the determinant and the Leibniz rule of
derivations yield the following properties of the Wronskian:
\begin{enumerate}
	\item[(1)] $W_{c\partial}(v)=c^{\binom {r+1} {2}}W_{\partial}(v)$ for each $c \in S$.
	\item[(2)] $W_{\partial}(vM)=(\det M)W_{\partial}(v)$ for each square matrix $M$ of size $r+1$ and entries in $k$.
\end{enumerate}

If $V\subseteq S$ is a free $k$-module of finite rank, denote
$W_\partial(V):=W_\partial(v)$, 
where $v:=[a_0 \,\, \cdots \,\, a_r]$, for $a_0,\dotsc,a_r \in S$ a ordered $k$-basis. Property $(2)$ above yields that $W_\partial(V)$ is well defined modulo multiplication by an invertible element of $k$.

Let $S:=k[X_0,X_1,X_2]$.
For each integer $d \geq 0$, let $S_d \subseteq S$ denote the free $k$-submodule of homogeneous polynomials of degree $d$, including $0$. A $k$-submodule $V \subseteq S$ is said to be \textit{homogeneous} of degree $d$ if $V\subseteq S_d$.

Let $\partial_{X_0},\partial_{X_1},\partial_{X_2}$ be the partial $k$-derivations of $S$ with respect to the variables $X_0,X_1,X_2$. A $k$-derivation $\partial$ of $S$ can be expressed in the form
$$
\partial = G_0\partial_{X_0} + G_1\partial_{X_1} + G_2\partial_{X_2},
$$
where $G_0,G_1,G_2 \in S$. We say that $\partial$ is \textit{homogeneous of degree} $d$ if $G_0,G_1,G_2$ are homogeneous of degree $d$.

Given $P \in S$, let
$$
\nabla(P):= \begin{bmatrix} \partial_{X_0}(P) & \partial_{X_1}(P) & \partial_{X_2}(P) \end{bmatrix}.
$$
If $Q \in S$ is another polynomial, let
$$
\partial_{P,Q}:= \det \begin{bmatrix}
\nabla(P) \\
\nabla(Q) \\
\nabla
\end{bmatrix}
:=\det \begin{bmatrix}
\partial_{X_0}(P) & \partial_{X_1}(P) & \partial_{X_2}(P) \\
\partial_{X_0}(Q) & \partial_{X_1}(Q) & \partial_{X_2}(Q) \\
\partial_{X_0} & \partial_{X_1} & \partial_{X_2}
\end{bmatrix} 
$$
$$
:=\begin{vmatrix}  \partial_{X_1}(P) & \partial_{X_2}(P) \\  \partial_{X_1}(Q) & \partial_{X_2}(Q) \end{vmatrix} \partial_{X_0} - \begin{vmatrix}  \partial_{X_0}(P) & \partial_{X_2}(P) \\  \partial_{X_0}(Q) & \partial_{X_2}(Q) \end{vmatrix} \partial_{X_1} + \begin{vmatrix}  \partial_{X_0}(P) & \partial_{X_1}(P) \\  \partial_{X_0}(Q) & \partial_{X_1}(Q) \end{vmatrix} \partial_{X_2}.
$$

Assume $k$ is a field. If $\partial:= G_0\partial_{X_0} + G_1\partial_{X_1} + G_2\partial_{X_2}$ is a homogeneous derivation of $S$ of degree $d$, then $\partial$ induces a section of $T_{\mathbb{P}^2_k}(d-1)$, or equivalently, a map
\begin{equation}\label{2.1.1}
\eta\colon \mathcal{O}_{\mathbb{P}^2_k}(1-d) \rightarrow T_{\mathbb{P}^2_k},
\end{equation}
where $T_{\mathbb{P}^2_k}$ is the tangent bundle of $\mathbb{P}^2_k$. We can describe $\eta$ in very concrete terms: the direction given by $\eta$ at a point $P \in \mathbb{P}^2_k$ is that of the line passing through $P$ and $(G_0(P) : G_1(P) : G_2(P))$, whenever these two points are distinct. This line is defined away from the closed subscheme $Z \subseteq \mathbb{P}^2_k$ cut out by the maximal minors of the matrix:
$$
\begin{bmatrix}
X_0 & X_1 & X_2 \\
G_0 & G_1 & G_2
\end{bmatrix}.
$$ 
Notice that these minors are $W_\partial([X_0 \,\, X_1]$,
$W_\partial([X_0 \,\, X_2])$ and $W_\partial([X_1 \,\, X_2])$. A point
$P \in Z$ 
is called a \textit{singularity of} $\eta$, or \textit{singular for}
$\eta$. 

The section of $T_{\mathbb{P}^2_k}$ is nonzero, or equivalently, 
$Z \neq \mathbb{P}^2_k$, whence a (singular) \emph{foliation}
of degree $d$ of 
$\mathbb{P}^2_k$, if $\partial$ is not a multiple of the Euler derivation:
$$
\varepsilon:= X_0\partial_{X_0} + X_1\partial_{X_1} + X_2\partial_{X_2}.
$$
The foliation induced by $\partial$ leaves invariant the plane curve $C$ defined
by $F=0$, for $F\in S$ homogeneous, if and only if 
$F|\partial(F)$. In other words, dualizing the map \eqref{2.1.1} 
we get the ``vector field'' 
$\eta^\vee: \Omega_{\mathbb{P}^2_k}^1 \rightarrow
\mathcal{O}_{\mathbb{P}^2_k}(d-1)$; the curve $C$ is invariant by
$\eta^\vee$ if there is a vector field $\eta' : \Omega_C^1 \rightarrow \mathcal{O}_{\mathbb{P}^2_k}(d-1)|_C$ making the following diagram commute:
$$
\begin{CD}
\Omega_{\mathbb{P}^2_k}^1|_C @>\eta^\vee|_C>>
\mathcal{O}_{\mathbb{P}^2_k}(d-1)|_C\\
@VVV @|\\
\Omega_C^1 @>\eta'>> \mathcal{O}_{\mathbb{P}^2_k}(d-1)|_C.
\end{CD}
$$
Also, there are finitely many singularities of the foliation on $C$ if and only if $\gcd(\partial, F)=1$. Here, if $\partial = G_0\partial_{X_0} + G_1\partial_{X_1} + G_2\partial_{X_2}$, then $\gcd(\partial, F)$ is, by definition, the greatest common divisor of $F$ and the maximal minors of the matrix
$$
\begin{bmatrix}
X_0 & X_1 & X_2 \\
G_0 & G_1 & G_2
\end{bmatrix}.
$$ 
When $\gcd(\partial, F)=1$ we say that
$\partial$ is \textit{prime} to $F$. When $F|\partial(F)$ we say that
$\partial$ is a $F$-derivation.
\end{subsct}

Let $k$ be an infinite field and $S:=k[X_0,X_1,X_2]$.
Let $F \in S$ be a nonconstant homogeneous polynomial. 

\begin{definition}\label{2.2.1}\setcounter{equation}{0} Let $G,H\in S$. We say that
  $G$ is \emph{projectively equivalent} to $H$ modulo $F$ in $S$ if
  there are $A\in S$ and $a\in k-\{0\}$ such that $G=aH+AF$. Let
$\partial_1$ and $\partial_2$ be two $F$-derivations. We say that $\partial_1$ and $\partial_2$ are 
\textit{projectively equivalent modulo} $F$, and we denote
$\partial_1\equiv_F\partial_2$, if there is $a \in k-\{0\}$ such that
for each linear form $L$ there are a homogeneous $k$-derivation $\partial$ and a homogeneous polynomial $N \in S$ satisfying
\begin{equation}\label{2.2.1.1}
L(\partial_1 -a\partial_2) = F\partial + N \varepsilon.
\end{equation}
\end{definition}

\begin{proposition}\setcounter{equation}{0}\label{2.2.2} Let $\partial_1$ and $\partial_2$ be two
  $F$-derivations. 
If $\partial_1\equiv_F\partial_2$ and
  $V \subseteq S_d$ is a homogeneous $k$-vector space, then the
  subscheme of $\mathbb{P}^2_k$ cut out by $W_{\partial_1}(V)$ on
  $F$ is the same as that cut out by $W_{\partial_2}(V)$.
\end{proposition}

\begin{proof} Indeed, fixing a basis of $V$, it follows from \eqref{2.2.1.1} and the multilinearity of the determinant that $F$ divides
$$
L^{\binom {r+1} {2}}(W_{\partial_1}(V)-a^{\binom {r+1} {2}}W_{\partial_2}(V))
$$
for each linear form $L$, where $\dim_k(V)=r+1$. Thus, since $k$ is infinite,
$$
W_{\partial_1}(V)\equiv_F a^{\binom {r+1} {2}}W_{\partial_2}(V)
$$
\end{proof}

\begin{subsct}\label{ramifsch}\setcounter{equation}{0}
{\bf Ramification schemes.} Let $k$ be a field of characteristic zero
and $S:=k[X_0,X_1,X_2]$.
Let $F \in S$ be a
nonzero homogeneous polynomial of degree $d > 0$. The equation $F=0$
defines a projective plane curve $C\subset \mathbb{P}^2_k$.

Let $V \subset S$ be a homogeneous $k$-vector space of degree $e$ and
dimension $r+1$, for certain integers $e > 0$ and $r\geq 0$. The space
$V$ induces a linear system of (projective) rank $r$ and degree $de$
on $C$. Let $R_F(V)$ denote the \textit{ramification scheme} of $C$
associated to $V$. On the open subset
$X_i \neq 0$, the ramification scheme is the locus cut out by $F$ and
the Wronskian $W_{\partial_{F,X_i}}(V)$, for $i=0,1,2$. 

Now, $R_F(V)$ might be infinite, indeed:
\end{subsct}

\begin{proposition}\label{2.3.1} {\rm (\cite{E1}, Prop.~7.8, p.~133)} The ramification
  scheme $R_F(V)$ is finite if and only if $F$ is square-free and
  the linear system $V$ is nondegenerate on each geometric irreducible
  component of
  $C$. 
\end{proposition}

In other words, denoting by $\bar{k}$ an algebraic closure
  of $k$, the ramification scheme $R_F(V)$ is finite if and only if the
  irreducible factors of the polynomial $F$ in
  $\bar{k}[X_0,X_1,X_2]$ are distinct and do not divide any nonzero
  element of $V\otimes_k\bar{k}$.

If $R_F(V)$ is finite, then $R_F(V)$ may be viewed
as a Cartier divisor of $C$. Before showing the next result we need introduce a few more concepts.

Let $P,Q \in S$ be nonconstant homogeneous polynomials with $\gcd(P,Q)=1$. Let $(P\cdot Q)$ denote the subscheme of $\mathbb{P}^2_k$ cut out by $P$ and $Q$, and $[P \cdot Q]$ the associated $0$-cycle. We will also view $(P \cdot Q)$ as a Cartier divisor of the curve cut out by $P=0$ or $Q=0$.

\begin{lemma}\label{2.3.2} Let $k$ be a field of characteristic zero. Let $P \in
  S:=k[X_0,X_1,X_2]$ be a nonzero homogeneous polynomial, and $C\subset
  \mathbb{P}^2_k$ the curve given by $P=0$. Let $V \subset S$ be a
  homogeneous $k$-vector space of dimension $r+1$, for a nonnegative
  integer $r$. 
Then the following four statements hold:
\begin{enumerate}
\item If $Q_1,Q_2 \in S$ are nonconstant and homogeneous, then
	$$Q_2\partial_{P,Q_1} \equiv_P Q_1\partial_{P,Q_2}.$$
\item For each nonconstant homogeneous polynomial $Q \in S$ prime to $P$, 
the ramification scheme $R_P(V)$ associated to $V$ on $C$ 
is finite if and only if $\gcd(W_{\partial_{P,Q}}(V),P)=1$, and in this case
\begin{equation}\label{2.3.2.1}
(W_{\partial_{P,Q}}(V)\cdot P) = R_P(V) + \binom {r+1} {2} (Q\cdot P)
\end{equation}
as Cartier divisors of $C$.
\item If $P$ is square-free, then $\gcd(\partial_{P,Q},P)=1$ for each nonconstant homogeneous polynomial $Q \in S$ prime to $P$.
\item Let $\partial$ be a $P$-derivation with
  $\gcd(\partial,P)=1$. If $P$ is square-free and 
$V$ is nondegenerate on each geometric irreducible component of $C$, then $\gcd(W_{\partial}(V),P)=1$.
\end{enumerate}
\end{lemma}

\begin{proof} Let us prove the first statement. Let $L$ be any nonzero
  linear homogeneous polynomial. We may assume without loss of
  generality that $L=X_2$.

For each homogeneous polynomial $Q$, let
\begin{align*}
\partial'_{P,Q}:= 
\begin{vmatrix} 
\partial_{X_0}(P) & \partial_{X_1}(P) & \varepsilon(P) \\
\partial_{X_0}(Q) & \partial_{X_1}(Q) & \varepsilon(Q) \\
\partial_{X_0}    & \partial_{X_1}    & \varepsilon
\end{vmatrix} :=&
\begin{vmatrix} 
\partial_{X_1}(P) & \varepsilon(P) \\
\partial_{X_1}(Q) & \varepsilon(Q)
\end{vmatrix}
\partial_{X_0} -
\begin{vmatrix} 
\partial_{X_0}(P) & \varepsilon(P) \\
\partial_{X_0}(Q) & \varepsilon(Q)
\end{vmatrix}
\partial_{X_1}\\
&+
\begin{vmatrix} 
\partial_{X_0}(P) & \partial_{X_1}(P) \\
\partial_{X_0}(Q) & \partial_{X_1}(Q)
\end{vmatrix} 
\varepsilon,
\end{align*}
where $\varepsilon$ is the Euler derivation. Notice that 
$X_2\partial_{P,Q}=\partial_{P,Q}'$. Let $q_1$ and $q_2$ be the degrees of $Q_1$ and $Q_2$, and set
\begin{align*}
Q_{X_0}:=&q_2Q_2\partial_{X_0}(Q_1) - q_1Q_1\partial_{X_0}(Q_2),\\
Q_{X_1}:=&q_2Q_2\partial_{X_1}(Q_1) - q_1Q_1\partial_{X_1}(Q_2).
\end{align*}
Since 
$$q_2Q_2\varepsilon(Q_1)-q_1Q_1\varepsilon(Q_2)=q_2Q_2q_1Q_1-q_1Q_1q_2Q_2=0,$$
we have 
\begin{align*}
{X_2}(q_2Q_2\partial_{P,Q_1}-q_1Q_1\partial_{P,Q_2})=
&q_2Q_2\partial'_{P,Q_1}-q_1Q_1\partial'_{P,Q_2}\\
=&pP 
\begin{vmatrix} 
Q_{X_0}          & Q_{X_1} \\
\partial_{X_0} & \partial_{X_1}
\end{vmatrix} 
+
\begin{vmatrix} 
\partial_{X_0}(P) & \partial_{X_1}(P) \\
Q_{X_0}	          & Q_{X_1}
\end{vmatrix}
\varepsilon,
\end{align*}
where $p$ is the degree of $P$. The proof of Statement 1 is complete.

To prove the remaining statements, we may assume $k$ is algebraically closed. Let us prove the second statement. Apply Statement 1 to $Q_1:=Q$ and $Q_2:={X_2}$. Then $X_2\partial_{P,Q}$ and $Q\partial_{P,X_2}$ are equivalent modulo $P$, and hence
\begin{equation}\label{2.3.2.2}
X_2^{\binom {r+1} {2}}W_{\partial_{P,Q}}(V)\equiv_P cQ^{\binom {r+1} {2}}W_{\partial_{P,X_2}}(V),
\end{equation}
for some $c \in k^\ast$. Now, $W_{\partial_{P,X_2}}(V)=0$ cuts out the subscheme $R_P(V)$ on $C$ in the open set $X_2 \neq 0$. Since $\gcd(Q,P)=1$, it follows from \eqref{2.3.2.2} that $R_P(V)$ is finite on $X_2 \neq 0$ if and only if $\gcd(W_{\partial_{P,Q}}(V),P)$ is a power of $X_2$. Applying the same argument to the open sets $X_0 \neq 0$ and $X_1 \neq 0$, it follows that $R_P(V)$ is finite if and only if $\gcd(W_{\partial_{P,Q}}(V),P)=1$. 

Furthermore, if $\gcd(W_{\partial_{P,Q}}(V),P)=1$, then \eqref{2.3.2.2} yields that
$$
\binom {r+1} {2} (X_2\cdot P) + (W_{\partial_{P,Q}}(V)\cdot P)= \binom {r+1} {2} (Q\cdot P) + (W_{\partial_{P,X_2}}(V)\cdot P).
$$
Thus, on the open set $X_2 \neq 0$ the equation \eqref{2.3.2.1} is
true. By analogy, 
\eqref{2.3.2.1} holds everywhere.

Now, let us prove that the third statement follows from the second. Since $k$ is infinite, we may assume that $P$ has no linear factor which is a linear combination of just two coordinate functions, say $X_0$ and $X_1$.  Let $Q \in k[X_0,X_1,X_2]$ be a nonconstant homogeneous polynomial prime to $P$ and let $V \subset k[X_0,X_1,X_2]$ be the $k$-vector subspace spanned by $X_0,X_1$. Since $\gcd(L,P)=1$ for each $L \in V$ and $P$ is square-free, Proposition $2.3.1$ implies that $R_P(V)$ is finite. So, it follows from the Statement 2 that $W_{\partial_{P,Q}}(V)$ is prime to $P$, and then $\partial_{P,Q}$ is prime to $P$.

Finally, let us prove the last statement. We may assume $P$ is
irreducible and not a multiple of $X_2$. By Proposition \ref{2.3.1} the scheme $R_P(V)$ is finite,
and thus $P$ does not divide $W_{\partial_{P,X_2}}(V)$ by Statement 2.
Since $\Omega_C^1$ is generically invertible, there is a dense open
subset $U \subset C$ such that
$\partial = \partial_{P,X_2}$
on $U$. It follows that
$$
(W_{\partial}(V)\cdot P) = (W_{\partial_{P,X_2}}(V) \cdot P)
$$
on $U$. Since $(W_{\partial_{P,X_2}}(V) \cdot P)$ is finite, $\gcd(W_{\partial}(V),P)=1$.
\end{proof}

\section{Infinitesimal families and limits}\label{families}

\begin{subsct}\setcounter{equation}{0}
{\bf Families and limits.} Let $k$ be an algebraically closed field of
characteristic zero, $k[[t]]$ the ring of formal power series and
$k((t)):=k[[t]][1/t]$ the 
field of formal Laurent series. Let $S:=k[X_0,X_1,X_2]$. Put
$S[[t]]:=S\otimes_k k[[t]]$ and $S((t)):=S\otimes_k k((t))$. 
View $S[[t]]$ (resp.~$S((t))$) with the induced grading, where $t$ has degree zero. A homogeneous element of $S[[t]]$ (resp.~$S((t))$) will be called a \textit{homogeneous power series} (resp.~\textit{Laurent series}).

For each $k$-vector space $V$, let $V[[t]]$ be the $k[[t]]$-module of power series on $t$ with coefficients in $V$. Given $P(t) \in V[[t]]$, denote by $P(0)$ the constant coefficient.

Let $V(t)\subseteq S[[t]]$ be a $k[[t]]$-submodule. We say that $V(t)$ is
\textit{saturated} if for each $P(t) \in S[[t]]$ such that $tP(t) \in V(t)$,
then also $P(t) \in V(t)$. Assume that $V(t)$ is a nonzero, saturated and
homogeneous $k[[t]]$-submodule of $S[[t]]$. Thus, since $k[[t]]$ is a
principal ideal domain, $V(t)$ is free, of rank $r+1$ for some integer
$r\geq 0$, and $V(t)$ has a $k[[t]]$-basis $[P_0(t) \,\, \cdots \,\, P_r(t)]$ 
of homogeneous power series whose constant coefficients 
are linearly independent over $k$. Denote by $V(0)$ the 
$k$-vector space spanned by $P_0(0), ..., P_r(0)$. 

We view $V(t)$ as a family of linear systems on the projective plane
with limit $V(0)$. 

Let $F(t)\in S_e[[t]]$ with $F(0)\neq 0$, where $e$ is a positive
integer. We view $F(t)=0$ as defining a family $C(t)$ 
of plane curves of degree $e$.  The
generic curve $C^\ast$ is cut out by $F^\ast=0$, 
which is $F(t)$ viewed as an element of
$S_e((t))$. 
Let $V(t)\subseteq S_d[[t]]$ be a nonzero saturated
$k[[t]]$-submodule, 
where $d$ is a positive integer. The $k[[t]]$-module $V(t)$ induces a
family of linear systems of degree $de$ on the family of curves $C(t)$. The generic linear system is induced by $V^\ast$,
which is just $V(t)[1/t]$, viewed as a $k((t))$-vector subspace of
$S_d((t))$.

Generally, we use the superscript ``$ ^\ast$'' to mean that a certain
family of ``objects'' should be considered as an ``object'' over $k((t))$.

For each closed subscheme $R \subset \mathbb{P}^2_{k((t))}$, we denote by
$$
\lim_{t \to 0}R\subseteq \mathbb{P}^2_k
$$
its schematic boundary in $\mathbb{P}^2_k$, called \emph{limit}.

Assume the generic ramification scheme $R_{F^\ast}(V^\ast) \subset \mathbb{P}^2_{k((t))}$ is finite. Denote by $R_F^0(V)$ the schematic boundary of $R_{F^\ast}(V^\ast)$ in $\mathbb{P}^2_k$, and denote by $[R_F^0(V)]$ the associated $0$-cycle. Our aim is to compute $[R_F^0(V)]$.
\end{subsct}

\begin{subsct}\setcounter{equation}{0}
{\bf $F(t)$-derivations.} Let $k$ be an algebraically closed field of characteristic 0  and $S:=k[X_0,X_1,X_2]$. 
Let $F(t) \in S_e[[t]]$ with $F(0) \neq 0$. Let $V(t) \subset S[[t]]$ be a nonzero, homogeneous, saturated $k[[t]]$-submodule of rank $r+1$, for some integer $r > 0$. 

To compute the schematic boundary $R_F^0(V)$ of the generic ramification scheme $R_{F^\ast}(V^\ast)$, we will consider homogeneous $k[[t]]$-derivations $\partial(t)$ of $S[[t]]$. Such derivations can be expressed in terms of the natural basis $\partial_{X_0},\partial_{X_1},\partial_{X_2}$ in the form
$$
\partial(t) = G_0(t)\partial_{X_0} + G_1(t)\partial_{X_1} + G_2(t)\partial_{X_2},
$$
where $G_0(t),G_1(t),G_2(t)$ are homogeneous power series with the
same degree, 
say $m$. Set
$$
\partial(0) := G_0(0)\partial_{X_0} + G_1(0)\partial_{X_1} + G_2(0)\partial_{X_2}.
$$
If $\partial(0)$ is not a multiple of the Euler derivation, then
$\partial(t)$ gives a family of singular foliations of the plane. We say $\partial(t)$ is a
$F(t)$-\textit{derivation} if $F(t)|\partial(t)(F(t))$. Geometrically,
the family of foliations given by $\partial(t)$ leaves invariant
the family of plane curves $C(t)$ defined by $F(t)=0$.

A simple example of a $F(t)$-derivation is
$$
\partial_{(F(t),H(t))}=\det \begin{bmatrix}
\partial_{X_0}(F(t)) & \partial_{X_1}(F(t)) & \partial_{X_2}(F(t)) \\
\partial_{X_0}(H(t)) & \partial_{X_1}(H(t)) & \partial_{X_2}(H(t)) \\
\partial_{X_0} & \partial_{X_1} & \partial_{X_2}
\end{bmatrix}, 
$$
where $H(t)$ is any homogeneous power series. If $H(t)$ has positive degree,
and $H^\ast$ and $F^\ast$ are coprime in $S((t))$, we can use
$\partial_{F^\ast,H^\ast}$ to compute $R_{F^\ast}(V^\ast)$ on the
generic curve $C^\ast$ defined by $F^\ast=0$. Indeed, assuming that
$R_{F^\ast}(V^\ast)$ is finite, 
by Lemma \ref{2.3.2}, its expression as a Cartier divisor on the generic curve is:
\begin{equation}\label{3.2.0.1}
R_{F^\ast}(V^\ast) = \big(W_{\partial_{F^\ast,H^\ast}}(V^\ast)\cdot F^\ast\big) 
- \binom {r+1} {2}(H^\ast \cdot F^\ast).
\end{equation}
\end{subsct}

\begin{remark}\label{3.2.1} To compute
$R_{F^\ast}(V^\ast)$ we can simply pick $H \in S$ homogeneous,
nonconstant and prime to $F(0)$. In this case, to compute 
$R_F^0(V)$ we can use Expression~\eqref{3.2.0.1}. However, the 
schematic boundary of $\big(W_{\partial_{F^\ast,H^\ast}}(V^\ast)\cdot F^\ast\big)$ 
will not necessarily be $\big(W_{\partial_{F(0),H}}(V(0))\cdot F(0)\big)$. 
In fact, the latter might not even make sense. It will not 
when an irreducible factor of $F(0)$ is multiple 
or divides a nonzero polynomial of $V(0)$. In any of these 
cases, this factor will also be a factor of $W_{\partial_{F(0),H}}(V(0))$.
\end{remark}

\section{Degenerations along a general direction}
\label{example0}

\begin{subsct}\label{3.3}\setcounter{equation}{0}
{\bf Reduced $F(t)$-derivations.} Let $k$ be an algebraically closed 
field of characteristic zero, and $S:=k[X_0,X_1,X_2]$. For each 
nonzero polynomial $P \in S$, write
$$
P=\prod_{i=1}^mE_i^{e_i},
$$
where $E_1,...,E_m$ are the irreducible factors of $P$. Let
$$
\overline{\nabla}(P):= \Big(\prod_{i=1}^m E_i\Big)\frac{\nabla (P)}{P} =
 \Big(\prod_{i=1}^m E_i\Big)\sum_{i=1}^m e_i \frac{\nabla (E_i)}{E_i}.
$$
Notice that
$$
\nabla(P)=(\prod_iE_i^{e_i - 1} ) \cdot \overline{\nabla}(P),
$$
for every $P \in S$.

Let $F(t) \in S[[t]]$ be a homogeneous power series of positive degree and nonzero constant coefficient $F(0)$. Let
$$
H(t):= (F(t) - F(0) )/t,
$$
and put
$$
\partial(t):= 
\begin{vmatrix}
\overline{\nabla}(F(0)) \\
\nabla(H(t))			    \\
\nabla
\end{vmatrix}.
$$
The derivation $\partial(t)$ is an $F(t)$-derivation. Indeed, 
first $\partial(t)(H(t))=0$. In addition, $\partial(t)(F(0))=0$, 
since $\nabla(F(0))$ is a multiple of the first row of the matrix
whose determinant is
$\partial(t)$. Thus
$$
\partial(t)(F(t)) = \partial(t)(F(0)) + t\partial(t)(H(t))=0.
$$
We say that $\partial(t)$ is the \textit{reduced} $F(t)$-derivation.
\end{subsct}

\begin{lemma}\label{3.3.1} 
With notation as above, if $\gcd(F(0),H(0))=1$ then $\gcd(\partial(0),F(0))=1$.
\end{lemma}

\begin{proof} Observe that, if $F(0)=\prod_iE_i^{e_i}$ is the factorization of $F(0)$, then
$$
\partial(0)=e_i(\prod_{j\neq i}^n E_j)\partial_{E_i,H(0)} + E_i\partial_i,
$$
where $\partial_i$ is a derivation. So, if $\gcd(E_i,H(0))=1$,
Lemma~\ref{2.3.2} yields $\gcd(\partial(0),E_i)=1$ for each $i$.
\end{proof}

Let $F(t):=\sum_{i \geq 0}F_it^i \in S[[t]]$  be homogeneous of positive
degree with $F_0\neq 0$ and $V \subset S[[t]]$ a nonzero, homogeneous, saturated
$k[[t]]$-submodule of rank $r+1$, where $r$ is a nonnegative integer. In
the next result we will see how to compute $R^0_F(V)$ in the case where $F_0$ has
multiple factors, at least when $F(t)\in S[[t]]$ is a deformation of
$F_0$ along a general direction, more precisely, when
$\gcd(F_0,F_1)=1$, and $V(0)$ is nondegenerate on each component of
the curve $C(0)$ given by $F_0=0$. 

If $F(t)$ is a deformation of $F_0$ along a general direction then
the generic curve $C^\ast$, given by $F^\ast=0$, is geometrically
reduced. This fact is proved below.

\begin{proposition}\label{3.3.2} Let $k$ be an algebraically closed field of 
characteristic zero, and $S:=k[X_0,X_1,X_2]$. 
Let $F(t):=\sum_{i \geq 0}F_it^i \in S[[t]]$ 
be a homogeneous power
series of positive degree with $F_0\neq 0$, and $C(t)$ the family of
plane curves it defines. Write
$$
F_0=\prod_{i=1}^mE_i^{e_i},
$$
where $E_1,...,E_m$ are the irreducible factors of $F_0$. Assume that
$\gcd(F_0,F_1)=1$. Then the
generic curve $C^\ast$ is geometrically
reduced. Furthermore, let $V(t) \subset S[[t]]$ be a nonzero, homogeneous, 
saturated $k[[t]]$-submodule of rank $r+1$, for 
$r\geq 0$. Assume that 
$V(0)$ is nondegenerate on each component of $C(0)$. Then 
$V^\ast$ is nondegenerate on each geometric component of
$C^\ast$, the generic ramification scheme $R_{F^\ast}(V^\ast)$ is finite, 
and the $0$-cycle of its limit $[R_F^0(V)]$ 
in $\mathbb{P}^2_k$ satisfies:
$$
[R_F^0(V)]= \sum_ie_i[R_{E_i}(V(0))] 
+\binom {r+1} {2}\sum_{i<j}(e_i+e_j)[E_i\cdot E_j]\\
+\binom {r+1} {2}\sum_i(e_i-1)[E_i\cdot F_1],
$$
where $R_{E_i}(V(0))$ is the ramification scheme of the linear 
system induced by $V(0)$ on the curve given by $E_i=0$ for each $i=1,\dots,m$.
\end{proposition}

\begin{proof} We may assume $V(t)$ is given. (One could let
  $V(t):=V[[t]]$ for a one-dimensional linear system $V$ generated by
  a homogeneous polynomial prime to $F_0$.) Let
$$
H(t):=(F(t)-F_0)/t.
$$ 
We have
\begin{equation}\label{FHE}
\partial_{F(t),H(t)}=\partial_{F(0),H(t)}=
\Big(\prod_{i=1}^m{E_i}^{e_i-1}\Big)\partial(t),
\end{equation}
where $\partial(t)$ is the reduced $F(t)$-derivation. In 
addition, for each $E_i$, as pointed out in the proof of Lemma~\ref{3.3.1},
\begin{equation}\label{3.3.2.3}
\partial(0)=e_i(\prod_{j\neq i} E_j)\partial_{E_i,F_1} + E_i\partial_i,
\end{equation}
where $\partial_i$ is a derivation. Since, by hypothesis, each $E_i$
does not divide either $F_1$ or a nonzero polynomial of $V(0)$, we
have $\gcd(W_{\partial(0)}(V(0)), E_i)=1$ by Proposition~\ref{2.3.1}
and Lemma~\ref{2.3.2}. It follows that
$\gcd(W_{\partial^\ast}(V^\ast),F^\ast)=1$. Furthermore, since
$\gcd(F_0,F_1)=1$, also $\gcd(F^\ast,H^\ast)=1$ and 
$\gcd(W_{\partial_{F^\ast,H^\ast}}(V^\ast),F^\ast)=1$. It follows now
from Lemma~\ref{2.3.2} that $R_{F^\ast}(V^\ast)$ is finite, and thus,
by Proposition~\ref{2.3.1}, that 
$C^\ast$ is geometrically reduced and
$V^\ast$ is nondegenerate on each geometric component of $C^\ast$.

By Lemma~\ref{2.3.2},
$$
(W_{\partial_{F^\ast,H^\ast}}(V^\ast)\cdot F^\ast) = R_{F^\ast}(V^\ast) + \binom {r+1} {2} (H^\ast\cdot F^\ast).
$$
From Expression~\eqref{FHE},
\begin{equation}\label{3.3.2.1}
R_{F^\ast}(V^\ast)= (W_{\partial^\ast}(V^\ast) \cdot F^\ast) + \binom {r+1} {2}\Big( \sum_i (e_i - 1)(E_i^\ast \cdot F^\ast) - (F^\ast \cdot H^\ast)\Big).
\end{equation}
Now, since $H(0)=F_1$,
$$
[\lim_{t \to 0}(E^\ast_i\cdot F^\ast)]=[E_i \cdot F_1] \ \ \textrm{ and } \ \ [\lim_{t \to 0}(F^\ast \cdot H^\ast)] = \sum_i e_i[E_i\cdot F_1].
$$
So, it follows from \eqref{3.3.2.1} that
\begin{equation}\label{3.3.2.2}
[R_{F}^0(V)] = [\lim_{t \to 0}(W_{\partial^\ast}(V^\ast)\cdot F^\ast)] - \binom {r+1} {2}\sum_i[E_i \cdot F_1].
\end{equation}

Now, since $\gcd(W_{\partial(0)}(V(0)), E_i)=1$ for each $E_i$, we
have
$$
[\lim_{t \to 0}(W_{\partial^\ast}(V^\ast)\cdot F^\ast)]= [W_{\partial(0)}(V(0))\cdot F(0)]= \sum_i e_i[W_{\partial(0)}(V(0))\cdot E_i].
$$
Using Formula~\eqref{3.3.2.3} and Lemma~\ref{2.3.2} we get
$$
[W_{\partial(0)}(V(0))\cdot E_i]=\binom {r+1} {2}\Big(\sum_{j \neq i}[E_j\cdot E_i] + [E_i\cdot F_1]  \Big) + [R_{E_i}(V(0))].
$$
Thus
$$
[\lim_{t \to 0}(W_{\partial^\ast}(V^\ast)\cdot F^\ast)] =\sum_i e_i[R_{E_i}(V(0))] + \binom {r+1} {2}\Big(\sum_i e_i[E_i\cdot F_1] + \sum_{i < j}(e_j+e_i)[E_i\cdot E_j]\Big).
$$
Combining the above expression with \eqref{3.3.2.2}, 
we get the desired expression for $[R_F^0(V)]$. 
\end{proof}

\section{Adaptations}
\label{section:adaptions}

Let $k$ be an algebraically closed field of characteristic zero and 
$S:=k[X_0,X_1,X_2]$.
Let $F(t) \in S[[t]]$ be a homogeneous power series with positive degree and nonzero constant coefficient $F(0)$. Let $V(t) \subset S[[t]]$ be a homogeneous, saturated $k[[t]]$-submodule of rank $r+1$ for some nonnegative integer $r$. 

If the generic ramification scheme $R_{F^\ast}(V^\ast)$ is finite, we
would like to compute its limit $0$-cycle $[R_F^0(V)]$ in
$\mathbb{P}^2_k$. As we saw in Lemma~\ref{2.3.2}, we can choose a
$F(t)$-derivation $\partial(t)$ such that $\gcd(W_{\partial^\ast}(V^\ast),
F^\ast)=1$ and this allows us to compute
$R_{F^\ast}(V^\ast)$. However, as pointed out in Remark~\ref{3.2.1},
the schematic boundary of $\big(W_{\partial^\ast}(V^\ast)\cdot
F^\ast\big)$ will not necessarily be $\big(W_{\partial(0)}(V(0))\cdot F(0)\big)$.

To remedy this we will consider modified derivations adapted to each factor of $F(0)$.

Indeed, to compute the limit of the ramification scheme
$R_{F^\ast}(V^\ast)$ in $\mathbb{P}^2_k$, we may change $\partial(t)$ to
any $F(t)$-derivation $\partial_1(t)$ such that the induced
$k((t))$-derivations $\partial^\ast$ and $\partial_1^\ast$ of $S((t))$
are projectively equivalent modulo $F^\ast$. The change is allowed because, by Proposition~\ref{2.2.2},
$$
W_{\partial_1^\ast}(V^\ast) \equiv_{F^\ast} c W_{\partial^\ast}(V^\ast),
$$
for some $c \in k((t))-\{0\}$.

We will actually consider something slightly more general, and for this we make the definitions below.

\begin{definition}\label{4.1.1}\setcounter{equation}{0} 
Let $F(t) \in S[[t]]$ be a homogeneous power series of positive degree
and nonzero constant coefficient $F(0)$. Let $E$ be an irreducible
factor of $F(0)$ and $\partial(t)$ a $F(t)$-derivation. We say that
$\partial(t)$ is \textit{adapted} to $E$ if $\gcd(\partial(0),E)=1$. We
say that a $F(t)$-derivation $\partial_1(t)$ is an \textit{adaptation} of
$\partial(t)$ to $E$ if $\partial_1(t)$ is adapted to $E$ and there is a
homogeneous power series $G(t) \in S[[t]]$ such that $\gcd(G(0),E)=1$ and
$\partial_1^\ast \equiv_{F^\ast} G^\ast \partial^\ast$.
\end{definition}

We do not know when such adaptations exist in general. 
But when they do, we may compute the limit $0$-cycle 
$[R_F^0(V)]$ using Theorem~\ref{4.1.4}, which is a 
simple consequence of Proposition~\ref{4.1.2} below.

\begin{proposition}\label{4.1.2} 
Let $k$ be an algebraically closed field. Let $F(t),G(t) \in S[[t]]$
be homogeneous power series of positive degree with $F(0)\neq 0$. Let $E_1,...,E_m$ be the irreducible factors of $F(0)$ and $e_1,...,e_m$ their respective multiplicities. Assume that, for each $i=1,...,m$, there are homogeneous power series $L_i(t),M_i(t) \in S[[t]]$ such that:
\begin{enumerate}
\item $L_i^\ast G^\ast$ is projectively equivalent to $M_i^\ast$ modulo $F^\ast$ in $S((t))$;
\item $L_i(0)M_i(0)$ is prime to $E_i$.
\end{enumerate}
Then $F^\ast$ and $G^\ast$ are coprime in $S((t))$ and
$$
[\lim_{t \to 0}(G^\ast \cdot F^\ast)]= \sum_{i=1}^m e_i \Big([M_i(0)\cdot E_i] - [L_i(0)\cdot E_i] \Big).
$$
\end{proposition}

\begin{proof} We may assume $G(0)\neq 0$. Also, we may work with an
  irreducible factor of $F(t)$ at a time, so we may assume $F(t)$ is irreducible.

We prove first that each of $G^\ast$, $L_i^\ast$ and
  $M_i^\ast$ for $i=1,\dots,m$ is coprime with $F^\ast$ in
  $S((t))$. Indeed, for each $i=1,\dots,m$, it follows from (1) that 
there are a homogeneous $A_i(t)\in S[[t]]$ with
  $A_i(0)\neq 0$, a power series $r_i(t)\in k[[t]]$ with $r_i(0)\neq 0$, and integers $m_i$ and $p_i$ such
 \begin{equation}\label{LGMF}
L_i(t)G(t)=t^{p_i}r_i(t)M_i(t)+t^{m_i}A_i(t)F(t).
\end{equation} 
If $p_i<0$, since $A_i(0)F(0)M_i(0)\neq 0$, we would have $m_i=p_i$. But
then $r_i(0)M_i(0)=-A_i(0)F(0)$ and thus $E_i$ would divide $M_i(0)$,
contradicting (2). Thus $p_i\geq 0$. Since $A_i(0)F(0)\neq 0$, also
$m_i\geq 0$. 

If $G^\ast$ and $F^\ast$ had a nontrivial common factor in
 $S((t))$, then $G(t)$ and $F(t)$ would have a common
 factor of positive degree in $S[[t]]$. Since $F(t)$ is irreducible,
 it would follow that $F(t)|G(t)$. But then it would follow from \eqref{LGMF} that $F(t)$ would
 divide $M_i(t)$ for each $i$, and hence $E_i|M_i(0)$, contradicting (2). 

Similarly, we show that $L_i^\ast$ and
  $M_i^\ast$ are coprime with $F^\ast$ in
  $S((t))$ for each $i=1,\dots,m$.

Let $B$ be the spectrum of $k[[t]]$ and let
  $\mathcal C\subset\mathbb P^2_B$ be the subscheme cut out by $F(t)=0$. Let
  $\pi\colon \mathcal C\to B$ be the projection. Then $\pi$ is flat, with 
  special fiber $C(0)$ of pure dimension 1. Let $\mathcal D$ be the subscheme
  cut out by $G(t)=0$ on $\mathcal C$. It is an effective Cartier
  divisor because $F^\ast$ and
  $G^\ast$ are coprime in $S((t))$. Similarly, the subschemes
  $\mathcal H_i$
  and $\mathcal K_i$ of $\mathcal C$ cut out by $L_i(t)=0$ and
  $M_i(t)=0$, respectively, are
  effective Cartier divisors for $i=1,\dots,m$.

For each $i=1,\dots,m$, let $\xi_i$ be the generic point
  of the irreducible primary subscheme of the special fiber $C(0)$ cut out by
  $E_i^{e_i}=0$. It follows from Equation~\eqref{LGMF} that
  $\mathcal D+\mathcal H_i=p_iC(0)+\mathcal K_i$ for $i=1,\dots,m$. 
And it follows from (2) that 
  $\xi_i\not\in\mathcal H_i+\mathcal K_i$ for $i=1,\dots,m$. 
Apply now \cite{E2}, Thm.~4.1, p.~1722.
\end{proof}

\begin{theorem}\label{4.1.4} 
Let $k$ be an algebraically closed field of characteristic zero. Let
$F(t) \in S[[t]]$ be a homogeneous power series of positive degree and
nonzero constant coefficient $F(0)$. Let 
$\partial(t)$ be a $F(t)$-derivation. Let $E_1,...,E_m$ be the irreducible
factors of $F(0)$ and $e_1,...,e_m$ their multiplicities. Let 
$V(t)\subset S[[t]]$ be a homogeneous, saturated $k[[t]]$-submodule of rank
$r+1$ for $r\geq 0$. Assume that, for each
$i=1,...,m$, the system induced by $V(0)$ on the curve given by $E_i=0$ is
nondegenerate. Assume as well that there are a positive integer $p$, 
and homogeneous power series $H_i(t),K_i(t) \in S[[t]]$ and an $E_i$-adapted
$F(t)$-derivation $\partial_i(t)$ for each $i=1,\dots,m$ such that:
\begin{enumerate}
\item $\partial_i^\ast \equiv_{F^\ast}H_i^\ast \partial^\ast$ in $S((t))$;
\item ${H_i^\ast}^p$ is projectively equivalent to $K_i^\ast$
  module $F^\ast$ in $S((t))$
\item $K_i(0)$ is prime to $E_i$.
\end{enumerate}
Then $W_{\partial^\ast}(V^\ast)$ and $F^\ast$ are coprime in $S((t))$ and
$$
[\lim_{t \to 0}(W_{\partial^\ast}(V^\ast)\cdot F^\ast)]= \sum_{i=1}^m e_i[W_{\partial_i(0)}(V(0))\cdot E_i] - 1/p\binom {r+1} {2}\sum_{i=1}^n e_i[K_i(0)\cdot E_i].
$$
\end{theorem}

\begin{proof} Set $G(t):=W_{\partial(t)}(V(t))^p$. Also, let 
$L_i(t):=K_i(t)^{\binom {r+1} {2}}$ and $M_i(t):=W_{\partial_i(t)}(V(t))^p$ for
$i=1,\dots,n$. Apply Proposition~\ref{4.1.2} and divide the resulting
equation by $p$.
\end{proof}

\section{Degenerations along a quasi-general direction}
\label{example}

\begin{theorem}\label{5} Let $k$ be an algebraically closed field of 
characteristic zero, and $S:=k[X_0,X_1,X_2]$.
Let $F(t):=\sum F_it^i \in S[[t]]$ be a homogeneous power
series of positive degree with 
$F_0\neq 0$, and $C(t)$ the family of plane curves it defines. Write
$$
F_0=\prod_{i=1}^mE_i^{e_i},
$$
where $E_1,...,E_m$ are the irreducible factors of $F_0$. 
Assume that $\gcd(E_i,F_1)=1$ for each $i$
such that $e_i>1$. Then the
generic curve $C^\ast$ is geometrically
reduced. Furthermore, let $V(t) \subset S[[t]]$ be a nonzero, homogeneous, 
saturated $k[[t]]$-submodule of rank $r+1$, for 
$r\geq 0$. Assume  that $V(0)$ is nondegenerate on each
component of $C(0)$. Then $V^\ast$ is nondegenerate on each geometric component of
$C^\ast$, the generic ramification scheme $R_{F^\ast}(V^\ast)$ is finite, 
and the $0$-cycle of its limit $[R_F^0(V)]$ 
in $\mathbb{P}^2_k$ satisfies:
$$
[R_F^0(V)]= \sum_ie_i[R_{E_i}(V(0))] 
+\binom {r+1} {2}\sum_{i<j}(e_i+e_j)[E_i\cdot E_j]
+\binom {r+1} {2}\sum_i(e_i-1)[E_i\cdot F_1].
$$
where $R_{E_i}(V(0))$ is the ramification scheme of the linear 
system induced by $V(0)$ on the curve given by $E_i=0$ for each $i=1,\dots,m$.
\end{theorem}

\begin{proof} As in the proof of Proposition~\ref{3.3.2}, we may
  assume $V(t)$ is given. Let $H \in k[X_0,X_1,X_2]$ homogeneous and prime to
  $F_0$. The $F(t)$-derivation 
$$
\partial_1(t):= \begin{vmatrix}
\partial_{X_0}(F(t)) & \partial_{X_1}(F(t)) & \partial_{X_2}(F(t)) \\
\partial_{X_0}(H) & \partial_{X_1}(H) & \partial_{X_2}(H)\\
\partial_{X_0} & \partial_{X_1}  & \partial_{X_2}
\end{vmatrix}
$$
is
adapted to each $E_i$ with $e_i=1$. Furthermore, the reduced $F(t)$-derivation
$$
\partial_2(t):=\frac{1}{E_1^{e_1-1}\cdots E_m^{e_m-1}}\begin{vmatrix}
\partial_{X_0}(F_0) & \partial_{X_1}(F_0) & \partial_{X_2}(F_0) \\
\partial_{X_0}(G(t)) & \partial_{X_1}(G(t)) & \partial_{X_2}(G(t)) \\
\partial_{X_0} & \partial_{X_1}  & \partial_{X_2}
\end{vmatrix},
$$
where $G(t):=(F(t)-F_0)/t$, is adapted to each $E_i$ with $e_i>1$.

We need to compare $\partial_1(t)$ to $\partial_2(t)$ to use
Theorem~\ref{4.1.4}. 
First observe that
$$
t\partial_2(t)=\frac{1}{E_1^{e_1-1}\cdots E_m^{e_m-1}} \begin{vmatrix}
\partial_{X_0}(F_0) & \partial_{X_1}(F_0) & \partial_{X_2}(F_0) \\
\partial_{X_0}(F(t)) & \partial_{X_1}(F(t)) & \partial_{X_2}(F(t)) \\
\partial_{X_0} & \partial_{X_1}  & \partial_{X_2}
\end{vmatrix},
$$
and hence, by Lemma~\ref{2.3.2},
$$
H^\ast\partial_2^\ast \equiv_{F^\ast} E_1\cdots E_m\partial_1^\ast
$$
as $k((t))$-derivations of $k[X_0,X_1,X_2]((t))$. Set
$$
\partial(t):= \prod_{e_i>1}E_i\partial_1(t)\quad\text{and}\quad \partial_3(t):=
H\partial_2(t).
$$
Since $\gcd(F_0,H)=1$, it follows that $\partial_3(t)$ is an adaptation
of $\partial(t)$ to each $E_i$ with $e_i>1$.

Set
$$
A_1:=\prod_{e_i=1}E_i \quad\text{and}\quad A_2:=\prod_{e_i>1}E_i
$$
It follows from Theorem~\ref{4.1.4}, for $p=1$, that 
$W_{\partial^\ast}(V^\ast)$ and $F^\ast$ are coprime, whence, since
$A_2^\ast$ and $F^\ast$ are coprime, $R_{F^\ast}(V^\ast)$ is finite by
Lemma~\ref{2.3.2}. As a consequence, $C^\ast$ is geometrically
reduced and $V^\ast$ is nondegenerate on each geometric component of
$C^\ast$ by Proposition~\ref{2.3.1}.

It follows as well from Theorem~\ref{4.1.4}, for $p=1$, that
\begin{equation}\label{1.3.1}
[\lim_{t \to 0}(W_{\partial^\ast}(V^\ast)\cdot F^\ast)] =
\sum_{e_i>1}e_i[W_{\partial_3(0)}(V(0))\cdot E_i] +
[W_{\partial(0)}(V(0))\cdot A_1] - \binom{r+1}{2}
\sum_{e_i>1}e_i[E_i \cdot A_1].
\end{equation}
Now, since $\partial_3(0)$ is equivalent to $(HA_1A_2/E_i)\partial_{E_i,F_1}$
modulo $E_i$ for each $i$, Lemma~\ref{2.3.2} implies that
\begin{equation}\label{1.3.2}
(W_{\partial_3(0)}(V(0))\cdot E_i) =\binom{r+1}{2}\Big( (HF_1\cdot E_i) +\sum_{j\neq
  i}(E_j\cdot E_i)\Big)+ R_{E_i}(V(0))
\end{equation}
for each $i$ with $e_i>1$. 
In the same way, since $\partial(0)$ is equivalent to
$(A_2F_0/E_i)\partial_{E_i,H}$ modulo $E_i$, we get
\begin{equation}\label{1.3.3}
(W_{\partial(0)}(V(0))\cdot E_i) = \binom{r+1}{2} (\frac{A_2F_0H}{E_i} \cdot E_i) + R_{E_i}(V(0))
\end{equation}
for each $i$ with $e_i=1$. Finally,
\begin{equation}\label{1.3.4}
(W_{\partial^\ast}(V^\ast)\cdot F^\ast) = R_{F^\ast}(V^\ast) +
\binom{r+1}{2} 
(A_2^\ast H^\ast \cdot F^\ast).
\end{equation}
So, taking the limit in Equation~\eqref{1.3.4} we get
\begin{equation}\label{1.3.5}
\lim_{t \to 0}(W_{\partial^\ast}(V^\ast)\cdot F^\ast) = \lim_{t \to
  0}R_{F^\ast}(V^\ast) +
\binom{r+1}{2}\Big( (A_2 \cdot F_1) + (H \cdot F_0)\Big)
\end{equation}
Thus, substituting \eqref{1.3.2}, \eqref{1.3.3} and \eqref{1.3.5} in
\eqref{1.3.1}, and taking associated $0$-cycles, the desired formula
follows.
\end{proof}

\begin{corollary}\label{5cor} Let $k$ be an algebraically closed field 
of characteristic zero, and $S:=k[X_0,X_1,X_2]$. 
Let $F(t):=\sum F_it^i \in S[[t]]$ be a 
homogeneous power series of positive degree with $F_0\neq 0$, and $C(t)$ the family of
plane curves it defines. Write
$$
F_0=\prod_{i=1}^mE_i^{e_i},
$$
where $E_1,...,E_m$ are the irreducible factors of $F_0$. Let $C_i$ be
the curve defined by $E_i=0$ for each $i$. Assume that
$\gcd(E_i,F_1)=1$ for each $i$ such that $e_i>1$. Then the generic
curve $C^\ast$ is geometrically reduced, and the limit 
of the dual plane curves of $C(t)$ satisfies:
$$
\lim_{t \to 0}(C^\ast)^\vee= \sum_ie_iC_i^\vee + \sum_{i < j}(e_i+e_j)[E_i\cdot E_j]^\vee + \sum_i(e_i-1)[E_i\cdot F_1]^\vee.
$$
\end{corollary}

\begin{proof} Apply Theorem~\ref{5} for $V(t):=V[[t]]$, where $V$ is a
  general pencil of lines and use \eqref{polar}.
\end{proof}

\begin{remark}\label{3.3.4} If $\gcd(E_i,F_1)=1$ for each $i$ such that $e_i>1$, then $F(t)$ is \textit{regular} in Katz's terminology. In \cite{Ka}, 
Thm.~3, p.~103, Katz gives a formula for $\lim_{t \to 0}(C^\ast)^\vee$ under the
regularity assumption. Our formula looks different from Katz's; it is
actually just simpler to present, as our formula is a special case of his.
\end{remark}

\section{Zeuthen families}
\label{section:zeuthen}

Let $k$ be an algebraically closed field of characteristic zero and $S:=k[X_0,X_1,X_2]$. 

\begin{lemma}\label{5.1.1} 
Let $F(t):= E^2A + F_1t + F_2t^2 + \cdots \in S[[t]]$ 
be a homogeneous power series of positive degree, where $A$ and $E$
are square-free and coprime. Let $E:=\prod_jE_j$ be the decomposition
in irreducible factors. For each $E_j$, let 
$B_j:=E^2A/E_j^2$, let $\Delta_{1,j}:=F_1$, and put 
$$
\Delta_{n+2,j}:=B_j^{n+1}F_{n+2}-\sum_{\substack{i+r=n+2}}\frac{\Delta_{i,j}'}{2}\cdot \frac{\Delta_{r,j}'}{2}
$$
for each integer $n\geq 0$, where
$\Delta_{i,j}':=\Delta_{i,j}/E_j$ for all $i,j$. Then, for each $E_j$ and each integer $n\geq 0$,
\begin{align*}
B_j^{2n+1}F(t)\equiv&\Big(E_jB_j^{n+1} + (\Delta_{1,j}'B_j^nt)/2+ \cdots + (\Delta_{i,j}'B_j^{n+1-i}t^{i})/2 + \cdots +  (\Delta_{n+1,j}'t^{n+1})/2 \Big)^2\\
&+B_j^n\Delta_{n+2,j}t^{n+2}\mod t^{n+3}.
\end{align*}
\end{lemma}

\begin{proof} Simple verification.
\end{proof}

\begin{definition}\label{5.1.2} 
We say that $F(t)$ is of \emph{type} $n$ for $E_j$ if $E_j$ divides 
$\Delta_{1,j},..., \Delta_{n-1,j}$ but does not divide $\Delta_{n,j}$.
\end{definition}

\begin{definition}\label{5.1.3} 
We call $\Delta_{i,j}$ the $i$-th \textit{discriminant} of $F(t)$
associated to $E_j$.
\end{definition}

\begin{remark}\label{5.1.4} 
When $E$ is irreducible, the family $C(t)$ given by $F(t)=0$ 
is a Zeuthen family of the first, second or 
third kind if and only if $F(t)$ is of type $1$, $2$ or $3$ for $E$, respectively,
cf.~\cite{vG}.  Also, if $F(t)$ is of type $1$, then $F(t)$ is a
special case of the $F(t)$ considered in Section \ref{example}.  
\end{remark}

\begin{theorem}\label{5.1.5} 
Let $k$ be an algebraically closed field of characteristic zero and
$S:=k[X_0,X_1,X_2]$.
Let $F(t):= E^2A + F_1t + \cdots \in S[[t]]$  be a homogeneous power series of positive
degree, where $A$ and $E$ are square-free and coprime, and $C(t)$ the
family of plane curves it defines.  
Let $E=E_1\cdots E_m$ be the decomposition in irreducible
factors. Assume the generic curve $C^\ast$ is geometrically reduced. 
Then for each $E_j$ there is an integer $n_j$ such that $F(t)$ is of
type $n_j$ for $E_j$. Furthermore, let 
$V(t) \subset S[[t]]$ be a saturated, homogeneous
$k[[t]]$-submodule of rank $r+1$, for some integer $r\geq 0$. Assume
that $V(0)$ is nondegenerate on each component of $C(0)$. Then 
the generic ramification scheme $R_{F^\ast}(V^\ast)$ is finite 
and the limit $0$-cycle $[R_F^0(V)]$ satisfies
\begin{align*}
[R_F^0(V)] =& 2\sum_{j=1}^m[R_{E_j}(V(0))] + [R_A(V(0))] + \binom{r+1}{2}[E^2\cdot A]\\
&+\sum_{j=1}^m\binom{r+1}{2}[\Delta_{n_j,j}\cdot E_j] 
- \sum_{j=1}^m\binom{r+1}{2}(n_j-2)[B_j\cdot E_j],
\end{align*}
with the $\Delta_{i,j}$ and the $B_j$ as defined in Lemma~\ref{5.1.1}.
\end{theorem}

\begin{proof} If $F(t)$ were not of type $n$ for $E_j$ for any $n>0$, then, by Lemma~\ref{5.1.1}, we would have 
$$
F(t)=\frac{\Big(E_jB_j + \sum_{i=1}^{\infty}(1/2)\Delta_{i,j}'B_j^{1-i}t^i\Big)^2} {B_j},
$$
and thus on the open set $B_j \neq 0$ the generic fiber would not be
reduced. Thus $F(t)$ is of type $n_j$ for each $E_j$ for a certain
$n_j$. 

There are now two cases to consider:

\textbullet\ \underline{$n_j=1$:} First notice that
$B_jF(t)=(E_jB_j)^2 + tB_jD(t)$, 
where $D(t):=(F(t)-E^2A)/t$. Let 
$$
\partial'_j(t):=B_j\partial_{B_jD(t), E_jB_j}.
$$
Then $\partial'_j(t)$ is a $F(t)$-derivation and is adapted to $E_j$ by Lemma~\ref{2.3.2}.

\textbullet\ \underline{$n_j \geq 2$:} Set $m_j:=n_j-2$. By Lemma~\ref{5.1.1} we have that
\begin{align*}
B_j^{2m_j+1}F(t)\equiv &\Big(E_jB_j^{m_j+1} + (\Delta_{1,j}'B_j^{m_j}t)/2+\cdots+(\Delta_{m_j+1,j}'t^{m_j+1})/2 \Big)^2\\
&+ B_j^{m_j}\Delta_{m_j+2,j}t^{m_j+2} \mod t^{m_j+3}.
\end{align*}
Define
\begin{align*}
Q_{1,j}(t):=& E_jB_j^{m_j+1} + (\Delta_{1,j}'B_j^{m_j}t)/2+\cdots +  (\Delta_{m_j+1,j}'t^{m_j+1})/2,\\
Q_{2,j}(t):=& (B_j^{2m_j+1}F(t) -Q_{1,j}(t)^2)/t^{m_j+2}.
\end{align*}
Since $F(t)$ is of type $n_j$ for $E_j$, we have $Q_{1,j}(t),Q_{2,j}(t)\in
S[[t]]$ and $E_j\nmid Q_{2,j}(0)$. Let
$$
\partial'_j(t):=B_j^{2m_j+1}\partial_{Q_{2,j}(t),Q_{1,j}(t)}.
$$
Then $\partial'_j(t)$ is a $F(t)$-derivation and is adapted to $E_j$ by Lemma~\ref{2.3.2}.

Let $H \in S$ be homogeneous and prime to $E^2A$. It follows from Lemma~\ref{2.3.2} that:
$$
H(\partial'_j)^\ast\equiv_{F^\ast}
\begin{cases}
B_j^3E_j\partial_{F^\ast, H^\ast}&\text{if }n_j=1,\\
B_j^{2(2m_j+1)}Q_{1,j}^\ast\partial_{F^\ast,H^\ast}
&\text{if }n_j\geq 2
\end{cases}
$$
as $k((t))$-derivations of $S((t))$.

Define
$$
\partial(t):=\partial_{F(t),H};
$$
$$
\partial_j(t):= H \partial'_j(t), \quad  H_j:=B_j^3E_j \quad  \textrm{ and } \quad  K_j(t):=B_j^5D(t), \quad  \textrm{ if } n_j=1;
$$
\[
\partial_j(t):= H \partial'_j(t),\quad H_j(t):=B_j^{2(2m_j+1)}Q_{1,j}(t) \quad \textrm{ and }\quad K_j(t):=B_j^{4(2m_j+1)}Q_{2,j}(t), \quad \textrm{ if } n_j\geq 2.
\]
The data $\partial_j(t)$, $H_j(t)$ and $K_j(t)$ satisfy all the
conditions of Theorem~\ref{4.1.4} for $p=2$. Thus
$W_{\partial^\ast}(V^\ast)$ and $F^\ast$ are coprime, and hence 
$R_{F^\ast}(V^\ast)$ is finite by Lemma~\ref{2.3.2}. Furthermore,
\begin{equation}
\label{5.1.5.1}
\begin{split}
[\lim_{t \to 0}(W_{\partial^\ast}(V^\ast)\cdot F^\ast)]=& 2\sum_{j=1}^m [W_{\partial_j(0)}(V(0))\cdot E_j] + \ [W_{\partial(0)}(V(0))\cdot A]\\
&- \binom{r+1}{2}\sum_{n_j=1}[B_j^5\Delta_{1,j}\cdot E_j] \\
&- \binom{r+1}{2}\sum_{n_j\geq 2}[ B_j^{4(2m_j+1)}B_j^{m_j}\Delta_{n_j,j}\cdot E_j ].
\end{split}
\end{equation}

We will now consider each term of \eqref{5.1.5.1}. From
Lemma~\ref{2.3.2}, since $\partial_j(0)=H\partial'_j(0)$, we have 
\begin{equation}\label{5.1.5.2}
[W_{\partial_j(0)}(V(0))\cdot E_j] =
\binom{r+1}{2}[HB_j^3\Delta_{1,j}\cdot E_j] + [R_{E_j}(V(0))]
\end{equation}
if $n_j=1$, whereas
\begin{equation}\label{5.1.5.3}  
[W_{\partial_j(0)}(V(0))\cdot E_j] = 
\binom{r+1}{2}[HB_j^{4m_j+2}\Delta_{n_j,j}\cdot E_j] +
[R_{E_j}(V(0))]
\end{equation}
if $n_j\geq 2$ (recall that $m_j=n_j-2$).
Also, since $\partial(0)=\partial_{E^2A,H}\equiv_A E^2\partial_{A,H}$ as $k$-derivations of $S$,
\begin{equation}\label{5.1.5.4}  
[W_{\partial(0)}(V(0))\cdot A] = \binom{r+1}{2}[E^2H\cdot A] + [R_A(V(0))].
\end{equation}

Finally,
\begin{equation}\label{5.1.5.5} 
(W_{\partial^\ast}(V^\ast)\cdot F^\ast) = R_{F^\ast}(V^\ast) + \binom{r+1}{2} (H^\ast \cdot F^\ast).
\end{equation}
So, taking limit $0$-cycles in \eqref{5.1.5.5} we get
\begin{equation}\label{5.1.5.6} 
[\lim_{t \mapsto 0}(W_\partial(V^\ast)\cdot F^\ast)] = [R_{F}^0(V)] + \binom{r+1}{2}[H \cdot AE^2].
\end{equation}
Thus, substituting \eqref{5.1.5.2}, \eqref{5.1.5.3}, \eqref{5.1.5.4}
and \eqref{5.1.5.6} into Equation~\eqref{5.1.5.1},
the stated formula for $[R^0_F(V)]$ follows. \end{proof}

\begin{corollary}\label{5.1.6} 
Let $k$ be an algebraically closed field of characteristic zero and
$S:=k[X_0,X_1,X_2]$.
Let $F(t):= E^2A + F_1t +
F_2t^2 + \cdots \in S[[t]]$ be a homogeneous power series of positive
degree, where $A$ and $E$ are square-free and coprime, and $C(t)$ be
the family of plane curves it defines. Let
$E=E_1\cdots E_m$ be the decomposition in irreducible factors.  Let $C_j$ be
the curve given by $E_j=0$ for each $j$, and $C_A$ that given by
$A=0$. If the generic curve $C^\ast$ is geometrically reduced, then the limit of the 
dual curves of the family $C(t)$ satisfies:
$$
\lim_{t \to 0}(C^\ast)^\vee =2\sum_{j=1}^mC_j^\vee + C_A^\vee + 2[E\cdot A]^\vee + \sum_{j=1}^m[\Delta_{n_j,j}\cdot E_j]^\vee- \sum_{j=1}^m(n_j-2)[B_j\cdot E_j]^\vee,
$$
where $B_j:=E^2A/E_j^2$, where $\Delta_{n_j,j}$ is the $n_j$-th
discriminant of $F(t)$ associated to $E_j$ and $n_j$ is the type of $F(t)$
for $E_j$, for each $j=1,\dots,m$.
\end{corollary}

\begin{proof} Apply Theorem~\ref{5.1.5} for $V(t):=V[[t]]$, where $V$ is a
  general pencil of lines, and use \eqref{polar}.
\end{proof}

\section{Acknowledgments} 
We are grateful to Jorge Vit\'orio Pereira and Israel Vainsencher for many 
discussions on the subject. We are specially grateful to
Steven Kleiman for insights, for pointing out
many references and for his comments on an earlier draft
of this paper. The first and second author are grateful to MIT for 
its hospitality during a visit of theirs,
when the seeds of this work were sown. The third author would like to
thank IMPA for its hospitality 
while this work was being finished.

\vspace{1cm}

\textsc{Instituto Nacional de Matem\'{a}tica Pura e Aplicada, Estrada
  Dona Castorina 110, 22460-320 Rio de Janeiro RJ, Brazil}

\textit{E-mail address:} \texttt{esteves@impa.br}

\vspace{0.2cm}
\textsc{Universidade Federal Fluminense, Instituto de Matem\'{a}tica e
  Estat\'{\i}stica, Rua Professor Marcos Waldemar de Freitas Reis,
  s/n, Campus Gragoat\'{a}, 24210-201 Niter\'{o}i RJ, Brazil}

\textit{E-mail address:} \texttt{nivaldomedeiros@id.uff.br}

\vspace{0.2cm}

\textsc{Universidade Federal da Paraíba, Centro de Ciências Exatas e da Natureza, Departamento de Matemática, Campus Universitário, 58051-900 João Pessoa PB, Brazil}

\textit{E-mail address:} \texttt{wallace@mat.ufpb.br}

\end{document}